\documentclass[twoside]{article}

\usepackage{blindtext} % Package to generate dummy text throughout this template

\usepackage[sc]{mathpazo} % Use the Palatino font
\usepackage[T1]{fontenc} % Use 8-bit encoding that has 256 glyphs
\usepackage{microtype} % Slightly tweak font spacing for aesthetics

\usepackage[english]{babel} % Language hyphenation and typographical rules

\usepackage[hmarginratio=1:1,top=32mm,columnsep=20pt]{geometry} % Document margins
\usepackage[hang, small,labelfont=bf,up,textfont=it,up]{caption} % Custom captions under/above floats in tables or figures
\usepackage{booktabs} % Horizontal rules in tables

\usepackage{lettrine} % The lettrine is the first enlarged letter at the beginning of the text

\usepackage{enumitem} % Customized lists
\setlist[itemize]{noitemsep} % Make itemize lists more compact

\usepackage{abstract} % Allows abstract customization
\usepackage{titlesec} % Allows customization of titles
\renewcommand\thesection{\Roman{section}} % Roman numerals for the sections
\renewcommand\thesubsection{\roman{subsection}} % roman numerals for subsections
\titleformat{\section}[block]{\large\scshape\centering}{\thesection.}{1em}{} % Change the look of the section titles
\titleformat{\subsection}[block]{\large}{\thesubsection.}{1em}{} % Change the look of the section titles

\usepackage{fancyhdr} % Headers and footers
\usepackage{titling} % Customizing the title section

\usepackage{hyperref} % For hyperlinks in the PDF

\pretitle{\begin{center}\Huge\bfseries} % Article title formatting
\posttitle{\end{center}} % Article title closing formatting
\title{Approximate  controllability of the semilinear reaction-diffusion equation governed by a multiplicative control}
\author{%
\textsc{M. Ouzahra}\\
\normalsize M2PA Laboratory, University of
Sidi Mohamed Ben Abdellah \\
P.O. Box 5206, Bensouda, F\`es, Morocco\\ % Your institution
\normalsize \href{mohamed.ouzahra@usmba.ac.ma}{mohamed.ouzahra@usmba.ac.ma} }
\date{}
\renewcommand{\maketitlehookd}{%
\begin{abstract}
In this paper we are concerned with  the approximate controllability of  a multidimensional semilinear reaction-diffusion equation governed by a multiplicative control, which is locally distributed in the reaction term. For a given initial state we  provide sufficient conditions on the desirable state to be approximately reached  within an arbitrarily small time interval. Moreover, in the case of a globally supported control, we prove the approximate  controllability within any time-interval given in advance which does not depend on the initial and target states.
 Our approaches are based on linear semigroup theory and some results on uniform approximation with smooth functions.
\end{abstract}

{\bf Keywords}: Semilinear parabolic equation, approximate controllability, bilinear control.}

\newtheorem{theorem}{Theorem}[section]
\newtheorem{corollary}{Corollary}
\newtheorem{lemma}[theorem]{Lemma}

\newtheorem{definition}[theorem]{Definition}
\newtheorem{remark}{Remark}
\newtheorem{example}{Example}

\newtheorem{proof}{Proof}

\begin{document}

\maketitle

\section{Introduction}
Our  goal in this paper is to study the global approximate controllability properties of the following semilinear Dirichlet boundary value problem
\begin{equation}\label{Eq}
\left\{
\begin{array}{ll}
  y_t(t) =   \Delta y(t) +v(x,t) {\bf 1}_O y(t)+f(t,y(t)),& \mbox{in} \; Q=\Omega\times (0,T_0) \\
  y = 0,  & \mbox{on} \;\Sigma=\partial\Omega\times (0,T_0)\\
  y(0)=  y_0\in L^2(\Omega), & \mbox{in} \;\Omega
\end{array}
\right.
\end{equation}
where  $ T_0>0,\, \Omega$ is a bounded open domain of $\mathbf{R}^d, d\ge 1$ with smooth boundary $\partial\Omega$ and $O$ is an  open subset of $\Omega$  with a characteristic function denoted by $ {\bf 1}_O$. The nonlinear term  $f: {\mathcal H}=[0,T_0] \times L^2(\Omega) \rightarrow L^2(\Omega)$ is  Lipschitz continuous in both variables, i.e.  there is a  constant $L>0$ such that
$$
\|f(t_1,y_1)-f(t_2,y_2)\|\le L(|t_1-t_2|+\|y_1-y_2\|),\;\; \forall (t_i,y_i)\in {\mathcal H}, i=1,2,
$$
where $\|\cdot\|$ refers to the conventional norm of $L^2(\Omega).$ Here, for each time $t\ge0$, the state $y(t)$ is given by the function $y(t)=y(\cdot,t) \in L^2(\Omega) $  and  $v(\cdot)$ is the control function which can be chosen in appropriate spaces. In terms of applications, equation (\ref{Eq}) provides practical description of various real problems such as chemical reactions, nuclear chain reactions, and biomedical models... (see \cite{and,ban88,ferr02,fri,kha10,per,roo,sal,tos} and the references therein).
 Research in  the  controllability of   distributed  systems by additive (linear) controls  have been the subject of several works (see for instance \cite{dou,fab,fer,fer99,fur96,glow,leb95,zua}). The question of  controllability of PDEs equations by multiplicative  (bilinear) controls has attracted many researchers in the context of  various type of  equations,  such as rod equation \cite{bal82,kim95}, Beam equation \cite{bea08}, Schr$\ddot{o}$dinger equation \cite{bea06,kim95,ners},   wave equation \cite{bal82,bea11,kha04,kha06,kha10,ouz14} and heat equation \cite{can10,can15,fer12,kha10,lin06,ouz15,ouz16}.  In \cite{can10}, the    approximate controllability properties have been derived for the one-dimensional  version  of (\ref{Eq}) for $f=0$ and  initial and target states   with finitely many changes of sign. The case where the support of the bilinear  control is allowed to depend  on time  has been discussed by Fern\`andez and Khapalov in \cite{fer12}. The exact controllability  of the  bilinear part of equation  (\ref{Eq}) with inhomogeneous Dirichlet conditions has been considered in \cite{lin06,ouz15}. However, the assumptions of \cite{lin06,ouz15} are not compatible when dealing with homogeneous Dirichlet conditions. In   \cite{kha03}, Khapalov studied the global  approximate controllability of  the semilinear convection-diffusion-reaction equation by bilinear controls while dealing  with nonnegative  initial and target states. In \cite{can15}, Cannarsa, Floridia and Khapalov have studied the global approximate controllability properties of the one dimensional version of (\ref{Eq}) with a time-independent nonlinear term when the initial and target states  admit no more than finitely many changes of sign.  In \cite{ouz16}, the question of multiplicative controllability of the bilinear variant of the system (\ref{Eq}) has been discussed when  the initial and target states $y_0, y^d$ are such that $y_0(x) y^d(x)\ge 0$  for  almost every $x\in \Omega$. We observe that, in the aforementioned papers, the control acts all over the evolution domain. Here,  our results  also includes the case of locally supported control.
 Moreover in the case of a globally distributed control, the multiplicative controllability of the reaction-diffusion semilinear  (\ref{Eq}) has been established in a time which   depends on the initial and target state. However in our case, we show that under a globally distributed control,  the time of  steering   can be taken independent  of the initial and target states.
  Furthermore, the steering control is constructed using an explicit approximation procedure, relying on Bernstein polynomials, combined with some density and approximation arguments.

The paper is organized as follows: In the next section, we present the main results. In the third section, we  provide some preliminary results that will be needed along the paper. The fourth section is devoted to the proofs of the main results.

\section{Main results }

We are interested in studying the approximate controllability of system (\ref{Eq}). More precisely,   we will first provide a locally supported  control  that can steer the system (\ref{Eq})  from its initial state $y_0$ to a final state $y(T) $ which is sufficiently close to the desirable state $y^d$  at a suitable time $ 0<T<T_0$  that  depends on the choice of $(y_0, y^d)$ and the  precision of steering $\epsilon>0.$
Moreover, in the context of a globally distributed  control, we will provide additional conditions that allow us to derive the approximate steering in a  uniform time-interval (i.e. the steering time is independent of the  initial and target states).\\
Everywhere below we will consider only non-zero initial states $y_0\in  L^2(\Omega),$ for which we  consider the set $\Lambda=\{x\in \Omega /\; y_0(x)\ne0\}$.

Our main results are as follows.
\begin{theorem}\label{thma}
Let $f$ be Lipschitz  in both variables, let $y_0\in L^2(\Omega)\setminus\{0\}$ be fixed and let $ y^d\in L^2(\Omega)$ be a desired state such that: (i) $\{x\in \Omega/\; y_0(x)\ne y^d(x)\} \subset O, $ a.e. and $a:=\ln(\frac{y^d}{y_0}){\bf 1}_{\Lambda\cap O} \in L^{\infty}(O), $ and
 (ii) for a.e. $x\in O, \; y_0(x) y^d(x)\ge 0$ and $y_0(x)=0\Longleftrightarrow y^d(x)=0$.\\
  Then for any $\epsilon>0, $ there are a time $ 0<T=T(y_0,y^d,\epsilon) <T_0$ and a static control $ v\in L^{\infty}(\Omega)$ such that for the respective solution to (\ref{Eq}), we have the following estimate:
\begin{equation}\label{estim*}
\|y(T)-y^d\| <\epsilon\cdot
\end{equation}

\end{theorem}

As a consequence of Theorem \ref{thma}, we have the following result.

\begin{corollary}\label{cor1}
 Let   $O=\Omega$ and let  $f$ be Lipschitz  in both variables. If   $y_0\in L^\infty(\Omega)$  and  $ y^d\in L^2(\Omega)$ are  such that $h=\frac{y^d}{y_0}{\bf 1}_\Lambda\in L^2(\Omega)$ and if  assumption (ii) of Theorem \ref{thma} holds,
  then for any $\epsilon>0, $ there are a time $ 0<T=T(y_0,y^d,\epsilon) <T_0$ and a static control $ v\in L^{\infty}(\Omega)$ such that for the respective solution to (\ref{Eq}), we have the  estimate (\ref{estim*}).

\end{corollary}

Theorem \ref{thma}   provides the basis for the following  controllability result within any a priori fixed time-interval.

\begin{corollary}\label{cor2}
Let  $O=\Omega,$ let $T\in (0,T_0)$ be fixed and let us set $Q_T=(0,T)\times \Omega$. We  assume that there is a positive constant $C$ such that for all $y\in L^2(\Omega)$ we have: $  |f(t,y)(x)|\le C |y(x)|,\; $ for all $t\in (0,T)$ and for a.e. $\; x\in \Omega$. Let $y_0\in L^2(\Omega)\setminus\{0\}$  and  $ y^d\in L^2(\Omega)$ satisfying  the condition $a=\ln(\frac{y^d}{y_0}){\bf 1}_{\Lambda} \in L^{\infty}(\Omega)$ and   the assumption (ii) of Theorem \ref{thma}. If in addition  one of the following hold \\

(a) $  y^d\in H^2(\Omega)$ and $   \frac{\Delta y^d}{y^d} {\bf 1}_\Lambda\in L^\infty(\Omega), $

or

(b) $y^d \in L^2(\Omega) $ and $y^d(x)\ge 0,\; $ a.e. $x\in\Omega$, \\

then  for any $\epsilon>0$, there exists $v\in L^\infty(Q_T)$ such that
\begin{equation}\label{estimT0}
\|y(T)-y^d\|<\epsilon\cdot
\end{equation}
\end{corollary}

\begin{remark}
$\bullet$   In \cite{kha03},  approximate controllability results have been established for initial and target states which  are not allowed to vanish in $\Omega$. Moreover, the one dimension version of equation (\ref{Eq}) has been studied in \cite{can10,can15}  with a nonlinearity which is independent of time  and also the points of "change of sign" of $y_0$ and $y^d$ are supposed finite. \\
   Note that, in one-dimensional case,    if the nonlinearity $f$ is time-independent and if  $y_0$ and $y^d$ have  opposite signs in a sub-interval of $\Omega, $ then our results are not applicable while those of $\cite{can10,can15}$ are applicable (provided the number of change of sign is finite and respect some order related to the maximum principle).

$\bullet$  As in Corollary \ref{cor1}, the result of Corollary  \ref{cor2} remains true if  the assumption
$\ln(\frac{y^d}{y_0}){\bf 1}_{\Lambda} \in L^{\infty}(\Omega)$ in Corollary  \ref{cor2} is replaced by $y_0\in L^\infty(\Omega)$  and $h=\frac{y^d}{y_0}{\bf 1}_\Lambda\in L^2(\Omega)$.

\end{remark}

\begin{example}
Let us consider the system (\ref{Eq}) with $d=2$,  $f(t,y)(x)=\tilde{f}(y(x)), \, $ a.e. $x\in \Omega:=(0,1)^2,$ where $\tilde{f}$ is Lipschitz from $\mathbf{R}$ to $\mathbf{R}$. Let $ O\subset \Omega,\; y_0=(x_1-x_2) {\bf 1}_O,$ a.e. $ x=(x_1,x_2) \in \Omega,$ let $y^d= k(x) y_0  {\bf 1}_O$ with  $k\in L^\infty(\Omega)$ and $ k(x)>0, $ a.e.  $ x\in O.$ We can observe that  $y_0$ and $y^d$ have the same sign a.e. in $O$. More precisely, $y_0$ and $y^d$ vanish on $\Gamma$ and  are positive in $\Gamma^+:=\{ (x_1,x_2)\in O/\, x_1>x_2\}$ and negative in $\Gamma^-:=\{ (x_1,x_2)\in O/\, x_1<x_2\}$. According to Theorem \ref{thma}, the initial state $y_0$ can be approximately steered to $y^d$ at a small time $T_1$ which depends on $y_0$ and $ y^d.$ Moreover, if the control is allowed to act  all over the evolution domain, i.e. $O=\Omega$  and if $y_0\ge 0,$ a.e., in $\Omega$, then by  Corollary \ref{cor2}, the steering time can be a  priori chosen without assuming that $\tilde{f}$ is Lipschitz, but only $|\tilde{f}(s)|\le C |s|,\; s \in \mathbf{R}\;  ($ for some $C>0).$ Here again, the  result of Corollary \ref{cor2} improves those in  the literature   requiring that the nonlinearity $f$ is Lipschitz and that the time of steering should depend on $y_0$ and $y^d$.
\end{example}
{\it Outline and main ideas for the proofs.}\\
The proofs of the main results in  Section \ref{proofM}  consist on establishing the estimate (\ref{estim*}) in several steps by distinguishing various cases on smoothness of the initial state and the considered static control. The main idea for the proof of Theorem \ref{thma} consists on looking for a time $T=T(y_0,y^d,\epsilon)$ depending on $(y_0, y^d)$ and the  precision of steering $\epsilon>0$, and a static control $v(x,t)=v_T(x)\in L^\infty(\Omega)$ depending on $T>0$  such that $e^{Tv_T}y_0=y^d $ a.e. in $\Omega$, and  showing that the respective solution to (\ref{Eq})  is such that $y(T)-y^d\to 0, $ as $T\to 0^+.$  This goal   will be  achieved by selecting a  static control $v(x,t) = v_T (x) $ that enables us to write
\begin{equation}\label{VCF*}
  y(T)-y^d=\int_0^Te^{\frac{T-s}{T}a(x)} \big ( A y(s) + f(s,y(s)) \big ) ds,
  \end{equation}
and showing that the right-hand side of this relation tends to $0$ as $T\to 0^+.$\\
The proof of Theorem \ref{thma} amounts to estimate the right-hand side in the above formula in order to prove that it can be made arbitrarily small as long as the static control $v_T(x)$ and  the steering time $T$ are well-chosen. At that point, smoothness assumptions are required on  the expected control and the respective  solution. Then, based on the variation of constants formula and linear semigroup theory, we can conclude by density and approximation arguments.\\
Finally, if the steering control is allowed to act in the whole domain (i.e. $O=\Omega$) we  show in Corollary \ref{cor2} that, given a priori prescribed time $T'>0,$ one can provide a second (time-dependent) control on $[T,T']$ (where $T=T(y_0,y^d,\epsilon)$ is given in Theorem \ref{thma}) that maintains the system at hand closer to a system  admitting the target state as an equilibrium, so that the  system's state remains close to the target state in the whole time-interval.

Note that the idea of exploiting the relation (\ref{VCF*}) to establish the approximate steering was first introduced  by Khapalov in \cite{kha03}  for initial and target state that have the same signs, and was exploited later in   \cite{can10,can15} for $d=1$ to study the case of initial and target states that change their sign in a finitely number of points. However, our methods differ from those of \cite{can10,can15,kha03} in the way to show that the right-hand side of (\ref{VCF*}) goes to $0$ as $T\to 0^+, $ when dealing with a locally supported control.

\begin{remark}

$\bullet $ It  \cite{kha03}, Khapalov has showed that the system (\ref{Eq}) can be steered  from any nonzero initial state $y_0$ into any desirable neighborhood of any nonzero target state   $y^d$ that has the same sign as $y_0$ at a time $T=T(y_0,y^d,\epsilon) > 0,$ which depends on the choice of $(y_0, y_d)$ and the desirable precision of steering $\epsilon$. Then, thanks to the spectral expansion of the semigroup associated to the pure diffusion part of (\ref{Eq}), the  strategy to achieve the desirable controllability result is  to select the multiplicative control in such a way that the corresponding trajectories of (\ref{Eq})  can be approximated by those associated with the pure diffusion (corresponding to $v = 0$ and $f = 0$) and the pure reaction (corresponding to $ A=0$ and $f=0$), while suppressing the effect of the nonlinearity which is considered as a disturbing  term.

$\bullet $ In \cite{can10},  Cannarsa and Khapalov established an approximate controllability property for the one dimensional  bilinear  equation (i.e. system (\ref{Eq}) with $d=1$ and $f=0$) when both the initial and target states are allowed to change their sign in a finite number of points respecting some roles that concorde with the maximum principle. Moreover, an implicit "continuation argument" was employed to justify the fact that one can always continue to move the points of sign change until their appropriate positions have been achieved. Moreover, for the same class of initial and target states as  in \cite{can10}, the authors  in \cite{can15} have extended  the approximate controllability results of \cite{can10} to semilinear equation like (\ref{Eq}) (with a tim-independent nonlinearity $f$).

%$\bullet $ In our main results, we employ a different  approach which allows us to improve the results in the literature (see for instance \cite{can10,can15,kha03}). Indeed,  the approximate controllability of system like (\ref{Eq})  in  a priori fixed time has not been  studied before. Moreover, the global controllability has been only established under globally distributed control. However, in one-dimensional case (say $\Omega =(0,1)$)    if the nonlinearity $f$ is time-independent and if  $y_0$ and $y^d$ have  opposite signs in a sub-interval of $\Omega, $ then our results are not applicable while those of $\cite{can10,can15}$ are applicable (provided the number of change of sign is finite and respect some order related to the maximum principle).

\end{remark}

\section{Preliminaries}

\subsection{Preliminary results  on linear semigroups and  evolution equations}
  Let us remind the reader that a one parameter family $S(t), \; t \ge 0,$ of bounded linear operators from a Banach space $X$ into $X$ is a semigroup  on $X$ if (i) $S(0) = I,  $  (the identity operator on $X)$ and (ii) $S(t + s) = S(t)S(s) $ for every $t, s\ge 0$. A semigroup $S(t)$ of bounded linear operators on $X$ is a $C_0-$ semigroup if in addition $\lim_{t\to 0^+} S(t)x = x$ for every $x\in  X. $ This property guarantees the continuity of the semigroup on $\mathbf{R}^+$. Moreover, one can show (see \cite{paz}, p. 4 ) that for every  $C_0-$ semigroup $S(t)$, there exist constants $\omega \ge 0$ and $M\ge 1$  such that
\begin{equation}\label{s-g}
 \|S(t)\| \le M e ^{\omega t},\; \forall t\ge 0\cdot
  \end{equation}
   If $\omega = 0$ and $M = 1,\, S(t)$ is called  a $C_0-$semigroup of contractions.\\
The linear operator $A$ defined by  $Ax = \lim_{t\to 0^+} \frac{S(t)x - x}{t} $ for $x\in X$ such that $\lim_{t\to 0^+} \frac{S(t)x - x}{t} $ exists in $X,$ is the infinitesimal generator of the $C_0-$semigroup $S(t). $ The   linear space  $ {\cal D} (A) := \{x\in X:\;  \lim_{t\to 0^+} \frac{S(t)x - x}{t} \in X \}$ is the domain of $A$.

$\bullet$ An infinitesimal generator of a $C_0-$semigroup of contractions is  dissipative, i.e., for every $y\in {\cal D}(A)$ there is  $y^*\in J(y)$ such that ${\cal R}e \langle Ay,y^*\rangle \le 0,$ where $J$ is the duality map from $X$ to  $X^{*}$, which, to each $y\in X,$ corresponds the  set  $J(y)$ of all $\phi\in X^{*}$ such that $\left< y,\phi\right >=\Vert y\Vert^{2}=\Vert \phi \Vert^{2}$, and where the dual $X^{*}$ of $X$ is the set of all bounded linear functionals  on $X$ and  $\left< y,\phi\right >$ is the duality pairing between $y\in X$ and $\phi\in X^{*}$. Conversely, if  $A$ is a densely defined closed linear operator such that both $A$ and its adjoint operator $A^*$ are dissipative, then $A$ is the infinitesimal generator of a $C_0-$semigroup of
contractions on $X$ (see \cite{paz}, pp. 14-15).

$\bullet$ The resolvent set $\rho(A)$ of an unbounded linear operator $A$ in a Banach space $X$ is the set of all complex numbers $\lambda$ for which;  $\lambda I - A$ is invertible, i.e. $(\lambda I - A)^{-1} $ is a bounded linear operator in $X.$ The
family $R(\lambda; A) := (\lambda I - A)^{-1},\;  \lambda \in \rho(A)$  is
called the resolvent of $A.$ The operator $R(\lambda;A)$  commutes with $A$ and $S(t),$ and  for all $y\in X,$ we have $ \lambda R(\lambda;A)y\to y,$ as $\lambda \to +\infty$. We also have that $ A R(\lambda;A) \in {\mathcal L}(X)$ and for all $ y\in {\mathcal D}(A);\; \lambda A R(\lambda;A)y \to Ay,$ as $\lambda \to +\infty$ (see  \cite{paz}, pp. 9-10).

$\bullet$ We have the following properties regarding   $C_0-$semigroups  (\cite{paz}, pp. 4-5)

(1) For every $x\in X, \, t\ge 0; \; \displaystyle \lim_{h\to 0} \frac{1}{h}\int_t^{t+h}   S(s)x ds = S(t)x.$

(2) For every $x\in X, \, t\ge 0; \; \displaystyle \int_0^t S(s)x ds \in {\cal D}(A) $ and
$ A ( \displaystyle\int_0^t S(s)x ds )  = T (t)x - x.$

(3) For every  $x\in {\cal D}(A)$ and $0\le s\le t;\;  S(t)x - S(s)x = \displaystyle \int_s^t S(\tau)Ax \:  d\tau  = \displaystyle \int_s^t AS(\tau)x \: d\tau.$

$\bullet$ From the above properties, one can  deduce that if $A$ is the infinitesimal generator of a $C_0-$semigroup $S(t),$ then
${\cal D}(A)$  (the domain of $A$) is dense in $X$ and $A$ is a closed linear operator.
 Moreover, according to Hille-Yosida's Theorem (see for instance \cite{paz}, p. 20), a linear operator $A$ is the infinitesimal generator of a $C_0-$semigroup $S(t)$ satisfying (\ref{s-g}) if and only if
(i) $A$ is closed and ${\cal D}(A)$ is dense in $X$, and (ii) the resolvent set $\rho(A)$ of $A$ contains the ray $(\omega,+\infty)$ and
$\|R(\lambda; A)^n\| \le \frac{ M}{(\lambda - \omega)^n} $ for $\lambda > \omega, \, n = 1, 2, ...$ In particular,   a closed operator $A$ with densely domain ${\cal D}(A)$ in $X$ is the infinitesimal  generator of a $C_0-$semigroup of contractions on $X$ if and only if  the resolvent set $\rho(A)$ of $A$ contains $ \mathbf{R}^+ $ and for all $\lambda>0; \; \|R(\lambda;A)\| \le \frac{1}{\lambda}$ (see \cite{paz}, p. 8).

$\bullet$ For $x \in {\cal D}(A);\; Ax= \frac{d^+S(t)x}{dt}|_{t=0} $ and that   $ y(t):= S(t)y_0   $ is differentiable and lies in ${\cal D}(A)$ for all $t>0, $ and is the unique solution of the Cauchy problem: $\dot{y}(t)=Ay(t), t>0, \; y(0)=y_0$. Moreover, for  every $y_0\in X;\; y(t)= S(t)y_0 $ is called mild solution of this Cauchy problem.

We now consider the nonhomogeneous  initial value problem
\begin{equation}\label{S1}
\left\{
\begin{array}{lll}
y_t \left( t\right) =Ay\left( t\right) +f\left( t,y\left( t\right)
\right) , & t\in \left[ 0,T\right] &  \\
\\
y\left( 0\right) =y_{0} &  &
\end{array}%
\right.
\end{equation}%
where $T>0,\; A$ is the infinitesimal generator of a $C_0-$semigroup $S(t)$ on $X$  and $f: [0,T]\times X\rightarrow X$ is a possibly non linear function.

Let us recall the notion  of weak solution from \cite{ball}.

\begin{definition}
 A function $y\in C\left( \left[ 0,T\right] ;X\right) $ is a  weak solution of (\ref{S1})  if for every $\varphi\in
{\cal D}\left(A^*\right) $ (the domain of the adjoint operator $A^*$ of $A$), the function $ t \mapsto \left\langle y\left( t\right) ,\varphi\right\rangle $ is absolutely continuous on $\left[0,T\right] $ and
$$
\frac{d}{dt}\left\langle y\left( t\right) ,\varphi\right\rangle =\left\langle
y\left( t\right) ,A^*\varphi\right\rangle +\left\langle f\left( t,y\left(
t\right) \right) ,\varphi\right\rangle,\ \ \mbox{for a.e. }\ \ t\in [0,T],
$$
 where $\left<\cdot,\cdot\right >$ is the duality pairing between the Banach space $X$ and its dual $X^*$.
 \end{definition}
A function $y\in C\left( \left[ 0,T\right] ;X\right) $ is a
weak solution of (\ref{S1}) on $\left[ 0,T\right] $ if and only if     $ f\left( .,y\left( .\right) \right) \in L^{1}(0,T;X)$  and  $y$ satisfies  the variation of constants formula (see \cite{ball}):
$$
y(t)=S(t)y_{0}+\int_{0}^{t}S(t-s)f(s,y\left( s\right) )ds,\; \forall t\in [0,T]\cdot
$$
Functions $y$ satisfying the above  formula are called "mild solutions" of the system (\ref{S1}). Moreover, the function  $y$ is a classical solution of (\ref{S1}) if $ y(t) \in {\cal  D}(A),$ for $t\in (0,T), \;y$ is continuous on $[0, T],\;  $  continuously differentiable on  $(0,T)$ and satisfies (\ref{S1}) (see  \cite{paz}, p. 126 $\&$ pp. 183-184). The  mild solution $y$ is  a strong solution  of (\ref{S1})   if it is  differentiable almost everywhere on $[0, T], \; y_t\in L^1(0,T; X) $ and satisfies (\ref{S1}) a.e.  on $ [0, T]$ (see \cite{paz}, p. 109).

The next  result  discusses the well-posedness  for the problem (\ref{S1}) in the case of  Lipschitz continuous functions $f.$

\begin{theorem} (\cite{paz}, p. 184).
  Let $ f: [0, T]\times X  \rightarrow X$ be continuous in $t$ on $[0,T]$ and
uniformly Lipschitz continuous  on $X.$ Then for every $y_0\in X$  the
system (\ref{S1}) has a unique mild solution $y\in {\cal C}([0,T]: X).$
Furthermore, the mapping $y_0 \mapsto y$ is Lipschitz continuous from $X$ into  ${\cal C}([0,T]: X)$.
\end{theorem}

A sufficient condition for the mild solution of (\ref{S1}) to be a classical solution is given next.

\begin{theorem}  (\cite{paz}, p. 187).
 Let $ f: [0, T]\times X  \rightarrow X$ be continuously differentiable on $[0, T]\times X.$
 Then the mild solution of (\ref{S1}) with $ y_0\in {\cal D}(A)$ is a classical
solution of  (\ref{S1}).
\end{theorem}

If $f$ is only  Lipschitz continuous, then the mild solution of (\ref{S1}) is not in general a classical  one. However, in the context of a reflexive space $X$, this may suffice to assure that the mild solution $y$ with initial state $y_0\in  {\cal D}(A)$ is a strong
solution. We have:
\begin{theorem} (\cite{paz}, p. 189). Assume that $X$ is a reflexive Banach space and that $ f: [0, T]\times X  \rightarrow X$ is Lipschitz  continuous in both variables. Then the  mild solution  $y$  of the initial value problem (\ref{S1}) with  $y_0\in {\cal D}(A)$ is  a strong solution of (\ref{S1}).
\end{theorem}

%Bernstein polynomials}

\subsection{Technical lemmas}

Let us give the following lemma which concerns the uniform approximation of continuous functions using Bernstein polynomials \cite{dav,dit,lor}.
\begin{lemma}\label{lemma3}
Let $u : [0,1] \to X$ be  a continuous  function from $[0,1]$ to a Banach space $(X,\|\cdot\|_X)$, and let $B_n(u)$  be  the $n$th Bernstein polynomial for $u$:
$$
B_n(u)(t)=\sum_{k=0}^n \left(
                                                                 \begin{array}{c}
                                                                   n \\
                                                                   k \\
                                                                 \end{array}
                                                               \right)  t^k(1-t)^{n-k}  u(\frac{k}{n}),\;\; n\ge 1.$$
Then the  sequence $B_n(u) $ tends uniformly to $u,$ i.e., $\sup_{t\in [0,1]} \|B_n(u)(t)-u(t)\|_X \to 0,$ as $n\to +\infty.$
\end{lemma}

\begin{proof} (\cite{dav}, pp. 108-112). Let $\epsilon>0.$ By Heine's theorem, we have:
\begin{equation}\label{unif-cont}
\exists  \eta>0;\; \forall  t,s \in [0,1],\;\; |t-s|<\eta \Rightarrow \|u(t)-u(s)\|_X<\frac{\epsilon}{2}\cdot
\end{equation}
Now, let us observe that $$
B_n(u)(t) -u(t) =\sum_{k=0}^n \left(
                                                                 \begin{array}{c}
                                                                   n \\
                                                                   k \\
                                                                 \end{array}
                                                               \right)  t^k(1-t)^{n-k} (u(\frac{k}{n})-u(t))\cdot
$$
Then in order to estimate this last sum, we separate the terms for which $|\frac{k}{n} - t| <\eta$ and  those for which $|\frac{k}{n} - t| \ge \eta.$
Thus, we can write
$$
B_n(u)(t) -u(t)=\Sigma_1 + \Sigma_2,
$$
where
$$
\Sigma_1=\sum_{|\frac{k}{n} - t| <\eta} \left(
                                                                 \begin{array}{c}
                                                                   n \\
                                                                   k \\
                                                                 \end{array}
                                                               \right)  t^k(1-t)^{n-k} (u(\frac{k}{n})-u(t)),
$$
and $$
\Sigma_2=\sum_{|\frac{k}{n} - t| \ge \eta} \left(
                                                                 \begin{array}{c}
                                                                   n \\
                                                                   k \\
                                                                 \end{array}
                                                               \right)  t^k(1-t)^{n-k} (u(\frac{k}{n})-u(t))\cdot
$$
Using (\ref{unif-cont}) it comes,
\begin{equation}\label{sigma1}
\|\Sigma_1\|_X <\frac{\epsilon}{2}.
\end{equation}
For the remaining terms, we will establish the following estimate
% (see \cite{dav}, pp. 109-111):
\begin{equation}\label{Bern}
\|\Sigma_2\|_X  \le \frac{M}{2\eta^2n}, \;\; \mbox{with} \; M=\sup_{t\in [0,1]}\|u(t)\|_X\cdot
\end{equation}
  Letting $b_k(t)=(\begin{array}{c} n \\
  k \\
  \end{array}
  ) t^k (1-t)^{n-k},$  for all $ t\in [0,1]$ and  $n\ge1,$ one  can see that $\sum_{k=0}^{n} b_k(t)=1$ and  $b_k(t) \in [0,1], $  for all $ t\in [0,1]$. Moreover,  by making use of the relation: $k {\cal C}^k_n=n{\cal C}^{k-1}_{n-1}, \, 1\le k\le n, $ we can show that:
                                                                $\sum_{k=0}^n (
                                                                 \begin{array}{c}
                                                                   n \\
                                                                   k \\
                                                                 \end{array}
                                                               ) b_k(t)=t$
                                                               and     $\sum_{k=0}^n (
                                                                 \begin{array}{c}
                                                                   n \\
                                                                   k \\
                                                                 \end{array}
                                                               )^2 b_k(t)=\frac{(n-1) t^2 +t}{n}$,
                                                                which leads to the formula
$\displaystyle\sum_{k=0}^n (t-\frac{k}{n})^2 b_k(t)= \frac{t(1-t)}{n},$
from which we deduce that
$$\eta^2 \displaystyle\sum_{|\frac{k}{n} - t| \ge \eta} b_k(t) \le \frac{t(1-t)}{n} \;\;\mbox{and}\;\;  \displaystyle\sum_{|\frac{k}{n} - t| \ge \eta} b_k(t)  < \frac{1}{4n \eta^2},$$
which gives (\ref{Bern}). Finally, taking $N\in \mathbf{N}$ such that $\frac{M}{2\eta^2N} < \frac{\epsilon}{2}$, we get
 $$
\sup_{0\le t \le 1}\|B_n(u)(t) -u(t)\|_X < \epsilon,\;\; \forall n\ge N\cdot
$$
\end{proof}

\begin{remark}
For all $n\ge 1,$ we have (\cite{dav}, pp. 112-113)
\begin{equation}\label{deriv-Bern}
B_n(u)'(t)=n\sum_{k=0}^{n-1}  \left(
                                                                 \begin{array}{c}
                                                                   n-1 \\
                                                                   k \\
                                                                 \end{array}
                                                               \right) t^k(1-t)^{n-1-k} (u(\frac{k+1}{n})-u(\frac{k}{n})),
\end{equation}
where  $B_n(u)'(t)$ is the derivative of $B_n(u)(t)$ with respect to $t.$
\end{remark}

%\subsection{Proof of Lemma \ref{lemma2}}

Let us now prove the following smoothness lemma.

\begin{lemma}\label{lemma2}
Let $\Omega$  be an open bounded set of $R^n, \; n\ge 1$. For all $h\in L^\infty(\Omega)$ such that $h\ge 0$, a.e. in $\Omega, $  there exists  $(h_r)\subset C^\infty(R^n)$ such that

 (i) $(h_r|_\Omega)$ is uniformly bounded with respect to $r$, (where $h_r|_\Omega$ designs the restriction of $h_r$ to $\Omega$),\\
 (ii) for all $r>0; \; h_r>0,\; $ a.e in $\overline{\Omega}, $\\
 and\\
 (iii) $h_r|_\Omega\to h $  in $L^2(\Omega), $ as $r\to 0^+$.
\end{lemma}

\begin{proof}
Let us extend $h$  by $0$  to $R^n$ so that the obtained extension, still denoted by $h$, lies in $L^2(R^n)\cap L^\infty(R^n)$.\\
Let us introduce the following function
$$
\phi(x)=  \left\{
  \begin{array}{ll}
   c \: e^{\frac{1}{\|x\|^2-1} } , & \hbox{if} \; \|x\|<1 \\
    0, & \hbox{if} \;\|x\|\ge1
  \end{array}
\right.
$$
where $c$ is a positive constant such that: $\int_{R^n}\phi =1.$ For all $r>0$, let $\phi_r(x)=r^{-n} \phi(\frac{x}{r}), \;$ a.e. $ x\in R^n$ and let $k_r$ be the convolution of $h$ with $\phi_r; \, k_r=\phi_r \ast h.$ This directly yields  $k_r\in {\mathcal C}^\infty(R^n),\, k_r\ge 0$ a.e. in $\overline{\Omega} $ and $ k_r\to h$  in $L^2(\Omega), $ as $r\to 0^+$ (see \cite{brez}, pp. 69-71). Moreover, for every $r>0$ and for a.e. $x\in \Omega$, we have
$$
\begin{array}{cccc}
  k_r(x)  & = & c\: r^{-n} \displaystyle\int_{B(x,r)}  h(s) e^{\frac{1}{\|\frac{x-s}{r}\|^2-1}}  ds \\
  \\
   & \le & \displaystyle c \: \|h\|_{L^\infty(R^n)} \: \displaystyle\int_{B(O,1)} ds.
 \end{array}
$$
In other word, the sequence  $(k_r)$ is uniformly bounded with respect to $r$.
We conclude that $h_r := k_r+r, r >0$ satisfies the claimed properties.
\end{proof}

\section{The proof of the main results}\label{proofM}

In the sequel, we consider  the  system (\ref{Eq}) on  the Hilbert state space  $H:=L^2(\Omega)$ equipped with its natural norm denoted by $\|\cdot\|$, and let us introduce the unbounded operator $A=\Delta$ with domain $ {\mathcal D}(A)=  H_0^1(\Omega) \cap H^2(\Omega),  $ endowed with the following graph norm: $\|y\|_{{\mathcal D}(A)}=(\|y\|^2+\|Ay\|^2)^{\frac{1}{2}}, \;  y\in {\mathcal D}(A).$ The operator $A$ generates a contraction  semigroup  $S(t)$ in $H$. Then $A$ is dissipative, i.e., ${\cal R} e \langle Az,z\rangle \le 0, \; \forall z\in {\cal D}(A)$, and we have $ \sup_{t\ge 0} \|S(t)\| \le 1$ and $ \sup_{\lambda> 0} \|\lambda R(\lambda;A)\|\le 1$.

\subsection{Proof of Theorem \ref{thma}}
Let us first observe that in the case where $a(x)=0,\, $ a.e. $ x\in \Omega,$  one can just use the null control, since $S(T)y_0+\displaystyle\int_0^T S(T-s)f(s,y(s)) ds \to y_0,\; $ as $T\to 0^+.$  Indeed, observing that $S(T)y_0 \to y_0,\; $ as $T\to 0^+,$ it suffices to show that:  $ \displaystyle\int_0^T S(T-s)f(s,y(s)) ds \to 0,\; $ as $T\to 0^+.$ \\
Let $T>0. $ For all $t\in [0,T]$ and $ y\in L^2(\Omega)$, we have
$$
\begin{array}{ccc}
\|f(t,y)\| &\le& \|f(t,y)-f(0,0)\|  + \|f(0,0)\| \\
\\
&\le& L (t+\|y\|) + \|f(0,0)\|,
\end{array}
$$
where $L$ is a Lipschitz constant of $f$. It follows that
\begin{equation}\label{*f}
  \|f(t,y)\| \le  L (T +\|y\|) + \|f(0,0)\|,\; \forall t\in [0,T]\cdot
\end{equation}
The mild solution  $y$ satisfies  the following variation of constants formula
\begin{equation}\label{vcf-0}
y(t)= S(t)y_0+ \displaystyle \int_0^t S(t-s) \big (  f(s,y(s)) \big ) ds,\, \forall t\in [0,T]\cdot
\end{equation}
Thus, using (\ref{*f})   we have
 \begin{equation}\label{a=0}
  \displaystyle\int_0^T \|f(s,y(s)) ds\|
  \le  T(L+ \|f(0,0)\|)+ L\displaystyle \int_0^t  \|y(s)\|   ds.
 \end{equation}
Then, it comes from (\ref{vcf-0})
$$
\|y(t)\| \le \|y_0\|+ T(L+ \|f(0,0)\|)+ L \displaystyle\int_0^t  \|y(s)\|   ds,
$$
which gives via Gronwall's inequality
$$
 \|y(t)\| \le \bigg ( \|y_0\| + T (L+ \|f(0,0)\|) \bigg )  e^{T L}\cdot
$$
This together with (\ref{a=0}) and the fact that $S(t)$ is a semigroup of contractions, gives
$$
\|\displaystyle\int_0^T S(T-s)f(s,y(s)) ds\|
  \le  T(L+ \|f(0,0)\|) + LT \bigg ( \|y_0\| + T (L+ \|f(0,0)\|) \bigg )  e^{T L},
$$
which gives the claimed result. Thus we shall  in the following assume $a(\cdot)\neq 0$. Moreover, for a time of steering $T>0$ (which is to be determined later) we consider the  control $v(x,t)=v_T(x):=\frac{a(x)}{T}. $   Then the system (\ref{Eq}) admits a unique mild solution $y(t)$ in $[0,T_0] $ in the state space $L^2(\Omega)$ (see \cite{paz}, Theorem 1.2, p. 184), which is given by the following  variation of constants formula
\begin{equation}\label{vcf2}
y(t)= S(t)y_0+ \int_0^t S(t-s) \big ( \frac{a(x)}{T}  y(s) + f(s,y(s)) \big ) ds,\, \forall t\in [0,T_0]\cdot
\end{equation}
 Furthermore since  $a\in L^\infty(\Omega)$ and $f$ is Lipschitz, it follows from  Gronwall's inequality \cite{drag,ye07} that the mapping $y_0 \mapsto y(t)$ is Lipschitz in H.\\
Now, from  the assumptions (i)-(ii) of Theorem \ref{thma}, we can derive that: $ e^{a}  y_0=y^d$. Indeed, for a.e. $x\in \Omega$ the formulae is obvious for  $x\in O\cap \Lambda.$ Moreover, the case $x\not\in O$ follows from the fact that: $\{x\in \Omega/\; y_0(x)\ne y^d(x)\} \subset O$ a.e. Now if $x\not\in \Lambda, $ then $y_0(x)=0$ and so by assumption (ii) of Theorem \ref{thma}, we  have $ y^d(x)=0$. Thus $ e^{a}  y_0=y^d$ a.e. in $\Omega.$ Having in mind this, the idea of the proof consists on remarking  that if we, formally, use the following formula
$$y(T)-y^d=\int_0^Te^{\frac{T-s}{T}a(x)} \big ( A y(s) + f(s,y(s)) \big ) ds,$$
then it suffices to show that the term in the right-hand side of the last relation tends to zero in $L^2(\Omega)$ as $T\to 0^+.$  To prove this, we need to show that the mild solution $y(t)$ of (\ref{Eq}) can be approximated by a  classical one, and then we conclude by an argument of density. We will distinguish several cases.

\subsubsection{ The case $a\in W^{2,\infty}(\Omega)$ and $y_0\in {\mathcal D}(A)$}
This first case consists on three steps.

{\bf Step 1.} In order to approximate the mild solution $y(t)$ with a classical one, we will approximate the continuous function $ t \mapsto f(t,y(t))$ with a $C^1-$function.\\
Without loss of generality, we  assume in the sequel  that $T_0=1.$ Also, for any $T\in (0,T_0),$ the  letter $C$ will be used to denote a generic positive constant (which is independent of $T$).\\
Since $S(t)$ is a semigroup of contractions, we deduce from (\ref{vcf2}) that
$$
\|y(t)\|  \le \|y_0\| + \frac{\|a\|_{L^\infty(\Omega)}}{T}  \int_0^T  \|y(s)\| ds +  \int_0^t  \|f(s,y(s))\| ds\cdot
$$
Then, using  (\ref{*f}), it comes:
$$
\|y(t)\| \le \|y_0\| + T(L+ \|f(0,0)\|) + \big ( \frac{\|a\|_{L^\infty(\Omega)}}{T} +L \big )  \int_0^t  \|y(s)\| ds,
$$
which, by Gronwall's inequality, leads to
$$
\|y(t)\| \le \bigg ( \|y_0\| + T(L+ \|f(0,0)\|) \bigg ) e^{\big ( \|a\|_{L^\infty(\Omega)} + TL \big )}\cdot
$$
Hence there exists a positive constant $C=C(\|a\|_{L^\infty(\Omega)})$  (which is independent of $T\in (0,1)$) such that
\begin{equation}\label{y(t)-borne}
\|y (t)\| \le C (1+ \|y_0\|),\; \forall t\in  [0,T]\cdot
\end{equation}
Let us consider the continuous function $F: t \mapsto f(t,y(t))$. Then, using (\ref{*f}) and (\ref{y(t)-borne}) and the fact that $T<1,$ we get
\begin{equation}\label{F(t)-borne}
\|F(t)\| \le C (1 +   \|y_0\|),\; \forall t \in [0,T],
\end{equation}
where  $ C=C(\|a\|_{L^\infty(\Omega)})>0$ is independent of $T.$\\
Let us show that $F$ is Lipschitz in $[0,T].$ For all $h, t\in [0,T]$ such that $t+h\in [0,T],$ we have
\begin{equation}\label{y(t)-lipsch}
\begin{array}{ccc}
  y(t+h)-y(t) & = & S(t+h)y_0-S(t)y_0+\displaystyle\int_0^h S(t+h-s) (\frac{a}{T} y(s) +F(s)) ds \\
  \\
  & + & \displaystyle \int_0^t S(t-s) \bigg (\frac{a}{T}( y(s+h)-y(s)) + (F(s+h)-F(s)) \bigg )ds\cdot
\end{array}
 \end{equation}
Let us  estimate the first and the last terms of the right side of (\ref{y(t)-lipsch}). For the first term, we have (since $y_0\in {\cal D}(A)$):
$$
\|S(t+h)y_0-S(t)y_0\|=\|\int_t^{t+h}  S(s)Ay_0  ds\| \le h \|Ay_0\|\cdot
$$
Moreover, from the definition of $F$, we have:
$$\|F(s+h)-F(s)\| \le L\big ( h+\|y(s+h)-\|y(s)\| \big ),$$
where $L$ is a Lipschitz constant of $f$.\\
Then using the two last estimates and inequalities (\ref{y(t)-borne})-(\ref{F(t)-borne}) and the fact that $S(t) $ is a contraction semigroup,  we derive from (\ref{y(t)-lipsch})
$$
\begin{array}{ccc}
\|y(t+h)-y(t)\| &\le& h \|Ay_0\| +hC(\frac{\|a\|_{L^\infty(\Omega)}}{T} +1)(1+\|y_0\| ) \\
\\
&+&  \displaystyle\int_0^t \bigg ( L h+ (\frac{\|a\|_{L^\infty(\Omega)}}{T}+L) \|y(s+h)-y(s)\| \bigg )ds\cdot
\end{array}
$$
Then using  $0< T< 1 < \frac{1}{T}, $ we deduce that:
$$
\|y(t+h)-y(t)\|\le h  \frac{\|Ay_0\| + C(1+\|y_0\|)+ L }{T}   +  \frac{\|a\|_{L^\infty(\Omega)}+L}{T} \int_0^t    \|y(s+h)-y(s)\| ds,
$$
where $C=C(\|a\|_{L^\infty (\Omega)})$ is independent of $T$, which by Gronwall's inequality gives the following estimate
$$
\|y(t+h)-y(t)\|\le \frac{C(1+ \|y_0\|_{{\mathcal D}(A)})}{T} h, \;\; \forall t\in [0,T],
$$
where $ C=C(\|a\|_{L^\infty(\Omega)})$ is independent of $T$.\\
Then using the last estimate and the fact that $f$ is  Lipschitz, this gives
$$
\begin{array}{ccc}
  \|F(t)-F(s)\| &\le & L\big (  |t-s| +\|y(t)-y(s)\|\big )\\
\\
  &\le & L\big ( 1+   \frac{C(1+ \|y_0\|_{{\mathcal D}(A)})}{T} \big ) |t-s|, \;\; \forall t, s\in [0,T]\cdot
\end{array}
 $$
This leads (for $0<T<1$) to
\begin{equation}\label{lipsch-F}
\|F(t)-F(s)\|\le \frac{M_1}{T}|t-s|, \;\; \forall t, s\in [0,T],
\end{equation}
where  $M_1=M_1(\|a\|_{L^\infty (\Omega)},\|y_0\|_{{\mathcal D}(A)}). $  \\
Then given  $\epsilon>0,$ we have for $\eta:=\frac{T\epsilon e^{-\|a\|_{L^\infty (\Omega)}}}{4 M_1}$
\begin{equation}\label{F-lipsch}
\forall  t,s \in [0,T],\;\; |t-s|<\eta \Rightarrow \|F(t)-F(s)\|<\frac{\epsilon e^{-\|a\|_{L^\infty (\Omega)}}}{4}\cdot
\end{equation}
Using Lemma \ref{lemma3},   we can uniformly approach  $F(t)$ on $[0,1]$ with the following sequence of polynomials: $$F_n(t)=\sum_{k=0}^n \left(
                                                                 \begin{array}{c}
                                                                   n \\
                                                                   k \\
                                                                 \end{array}
                                                               \right)  t^k(1-t)^{n-k} F(\frac{k}{n}),\;\; n\ge 1\cdot$$
From the expression of $F_n(t)$, we  have by virtu of (\ref{F(t)-borne})
\begin{equation}\label{Fn}
\sup_{t\in [0,T]} \|F_n(t)\| \le C (1+  \|y_0\|),
\end{equation}
where  $ C=C(\|a\|_{L^\infty(\Omega)})>0$ is independent of $T\in (0,1).$\\
Moreover,  (\ref{sigma1}) $\&$ (\ref{Bern})  combined with  (\ref{F(t)-borne}) $\&$ (\ref{F-lipsch}), gives
\begin{equation}\label{Fn-F}
\|F_n(t) -F(t)\|\le \frac{\epsilon e^{-\|a\|_{L^\infty (\Omega)}}}{4} + \frac{C(1+  \|y_0\|)}{2 n\eta^2}, \;\; \forall t\in [0,T],
\end{equation}
where  $ C=C(\|a\|_{L^\infty(\Omega)})>0$ is independent of $T\in (0,1).$\\
Thus (recall that $\eta:=\frac{T\epsilon e^{-\|a\|_{L^\infty (\Omega)}}}{4 M_1}$)
$$
\|F_n(t) -F(t)\|\le  \frac{\epsilon e^{-\|a\|_{L^\infty (\Omega)}}}{4} +\frac{M_2}{n T^2\epsilon^2}, \;\; \forall t\in [0,T],
$$
where $M_2=M_2(\|a\|_{L^\infty (\Omega)},\|y_0\|_{{\mathcal D}(A)})$  is independent of $T$ and $n$.\\
Hence we have
$$
\|F_n(t) -F(t)\| <  \frac{\epsilon e^{-\|a\|_{L^\infty (\Omega)}}}{2},\; \forall t\in [0,T],
$$
whenever
\begin{equation}\label{n*T}
n T^2 > \frac{4 M_2e^{\|a\|_{L^\infty (\Omega)}}}{\epsilon^3}\cdot
\end{equation}
Let $y_n(t)$ be the  solution of the system:
\begin{equation}\label{Eqyn}
 \frac{d}{dt}y_n(t)=A y_n(t) +\frac{a(x)}{T} y_n(t)+F_n(t),\; t\in (0,1), \; y_n(0)=y_0\cdot
\end{equation}
Then  we have
\begin{equation}\label{yn(t)}
y_n(t)= S(t)y_0+ \int_0^t S(t-s) \big ( \frac{a(x)}{T}  y_n(s) + F_n(s) \big ) ds, \;\; \forall t\in [0,1]\cdot
\end{equation}
Thus, using   (\ref{Fn}) and the Gronwall inequality, we get
$$
\|y_n(t)\| \le C (1+\|y_0\|), \;\; \forall t\in [0,T],
$$
for some  constant $C=C(\|a\|_{L^\infty (\Omega)})$ which is independent of $T$ and $n$.\\
Moreover, it follows from (\ref{vcf2}) and (\ref{yn(t)}) that:
$$
\|y_n(t)-y(t)\| \le \int_0^t \frac{\|a\|_{L^\infty (\Omega)}}{T}\| y_n(s) -y(s)\|  ds   +  \int_0^T \|F_n(s)-F(s)\| ds, \;\; \forall t\in [0,T]\cdot
$$
Then under (\ref{n*T}) we have, $$\|y_n(t)-y(t)\|\le \int_0^t \frac{\|a\|_{L^\infty (\Omega)}}{T}  \|y_n(s)-y(s)\| ds + \frac{T\epsilon e^{-\|a\|_{L^\infty(\Omega)}}}{2},$$
which gives via Gronwall inequality
\begin{equation}\label{estim-error1}
 \|y(t)-y_n(t)\| < \frac{T \epsilon}{2},\;\; \forall t\in[0,T]\subset [0,1],
\end{equation}
and hence
$$
\sup_{t\in[0,T]} \|y(t)-y_n(t)\|< \frac{\epsilon}{2}\cdot
$$
{\bf Step 2.} Here, we will establish an upper bound for the solution $y_n(t)$ of (\ref{yn(t)}) with respect to the graph norm.\\
Since $y_0\in {\mathcal D}(A)$ and $F_n\in C^1([0,1];L^2(\Omega)),$ we have that $y_n(t)$ is a classical solution (see \cite{paz}, Theorem 1.5,  p. 187). Then for all $t>0$, we have $y_n(t)\in D(A)$ and
\begin{equation}\label{yn-PDE}
\frac{d}{dt} \langle y_n(t),\phi \rangle =   \langle y_n(t),\frac{a(x)}{T}\phi \rangle+ \langle A y_n(t) + F_n(t),\phi \rangle, \; \forall t\in (0,1), \; \forall \phi\in H\cdot
\end{equation}
In other words, $y_n(t)$ is a weak solution of (\ref{Eqyn}).  We know that  $Ay_n \in L^2(0,1;L^2(\Omega))$ (see for instance \cite{eva}, pp. 360-361). Hence  $y_n(t)$  (see \cite{bal78,bal1})  satisfies the following variation of constants formula:
\begin{equation}\label{yn-vcf}
y_n(t)=  e^{t\frac{a(x)}{T}} y_0+\int_0^te^{(t-s)\frac{a(x)}{T}} ( A y_n(s) + F_n(s))ds,\; \forall t\in [0,1]\cdot
\end{equation}
In particular, we have
\begin{equation}\label{vcf1}
y_n(T)-y^d=\int_0^Te^{\frac{T-s}{T}a(x)} (A y_n(s)+F_n(s))ds\cdot
\end{equation}
Applying the bounded operator $A_\lambda = \lambda R(\lambda;A) A$ to (\ref{yn(t)}), we get
$$
\begin{array}{lll}
   A_\lambda y_n(t)  &=& S(t)A_\lambda y_0+ \displaystyle \int_0^t A_\lambda S(t-s) (\frac{a(x)}{T}  y_n(s) + F_n(s)) ds\\
     &=&  S(t)A_\lambda y_0 + \displaystyle\frac{1}{T}  \int_0^t A_\lambda S(t-s)  (a(x)  y_n(s)) ds + \int_0^t \lambda R(\lambda,A) S'(t-s) F_n(s) ds\\
  &    = &S(t) A_\lambda y_0 + \displaystyle\frac{1}{T}  \int_0^t A_\lambda S(t-s)  (a(x)  y_n(s)) ds -
\int_0^t \frac{d}{ds} \big ( \lambda R(\lambda;A) S(t-s)  F_n(s) \big ) ds +\\
&& \int_0^t \lambda R(\lambda;A) S(t-s)  F'_n(s)  ds\\
&=& S(t)A_\lambda y_0 + \displaystyle\frac{1}{T}  \int_0^t A_\lambda S(t-s)  (a(x)  y_n(s)) ds +\\
&& \lambda R(\lambda;A) \big ( S(t) F_n(0)-F_n(t) \big ) + \int_0^t \lambda R(\lambda;A) S(t-s) F'_n(s)ds\cdot
\end{array}
$$
Since $y_n(t)\in {\mathcal D}(A)$ and $a\in W^{2,\infty}(\Omega), $ we also have  $ay_n(t)\in {\mathcal D}(A)$ for all $t\in [0,1].$ Hence, using the properties of the semigroup and the resolvent associated to $A,$ we deduce from the above expression that
$$\|A_\lambda y_n(t)\| \le  \|Ay_0\| +  \displaystyle\frac{1}{T} \int_0^t  \|   A (a(x)  y_n(s)) \| ds + \| F_n(t)\| +\|F_n(0)\| + \int_0^t \| F'_n(s)\| ds\cdot
$$
Then,  letting $\lambda\to +\infty,$ we get
\begin{equation}\label{5*}
\|A y_n(t)\| \le  \|Ay_0\| +  \frac{1}{T}  \int_0^t  \| A  (a(x)  y_n(s)) \| ds + \| F_n(t)\| +\|F_n(0)\| + \int_0^t \| F'_n(s) \| ds,
\end{equation}
where the constant $C=C(\|a\|_{W^{2,\infty}(\Omega)})$ is independent of $T$ and $n.$\\
Let us now study the terms of right hand of inequality (\ref{5*}). We have by (\ref{deriv-Bern}) that
$$
F'_n(t)=n\sum_{k=0}^{n-1}  \left(
                                                                 \begin{array}{c}
                                                                   n-1 \\
                                                                   k \\
                                                                 \end{array}
                                                               \right) t^k(1-t)^{n-1-k} (F(\frac{k+1}{n})-F(\frac{k}{n})),
$$
which by (\ref{lipsch-F}) gives
\begin{equation}\label{Fn'}
\sup_{0\le t \le T}\|F'_n(t)\|\le  \frac{M_1}{T}, \; (M_1=M_1(\|a\|_{L^{\infty}(\Omega)},\|y_0\|_{{\mathcal D}(A)}))\cdot
\end{equation}
Moreover, for every $y$ in $ H^2(\Omega)$, we have the following  second order Leibniz rule
\begin{equation}\label{A(ay)}
\Delta(a y)=  y  \:\Delta a  +2 \nabla a \cdot \nabla y + a \: \Delta y, \; \mbox{a.e. in } \;  \Omega,
\end{equation}
and
\begin{equation}\label{Gre-Pt}
\|\nabla y \| \le C\|\Delta y\|,
\end{equation}
for some constant $C>0$ which  depends only on $\Omega.$\\
Taking into account (\ref{Gre-Pt}) and the fact that $a\in W^{2,\infty}(\Omega)$, we derive from (\ref{A(ay)})
\begin{equation}\label{GP}
\int_0^t \|A  ( a(x)  y_n(s) )\| ds \le  \|\Delta a\|_{L^\infty(\Omega)} \int_0^t\|y_n(s)\| ds +C \int_0^t \|Ay_n(s) \| ds,\, \forall t\in [0,T],
\end{equation}
where $C=C(\|a\|_{W^{2,\infty}(\Omega)})$ is independent of $T$ and $n.$\\
Then reporting  (\ref{Fn}), (\ref{Fn'}) and (\ref{GP}) in (\ref{5*}),  we deduce, via  Gronwall's inequality
\begin{equation}\label{ynDA}
\|y_n(t)\|_{D(A)}\le M_3,\; \forall t\in [0,T],
\end{equation}
where $M_3=M_3(\|a\|_{W^{2,\infty}(\Omega)},\|y_0\|_{D(A)})$  is independent of $T$ and $n.$

{\bf Step 3.} We now show that, for $n$ large enough, one can choose  $T$ small enough so that $y_n(T)$ (and so is $y(T)$) approaches $y^d$ with any a priori fixed  precision.\\
Using the estimates (\ref{Fn}) and (\ref{ynDA}), we get from the relation (\ref{vcf1})
\begin{equation}\label{estim-error2}
  \|y_n(T)-y^d\|  \le  M_4 T
\end{equation}
for some   constant $M_4=M_4(\|a\|_{W^{2,\infty}(\Omega)},\|y_0\|_{D(A)})$ which is independent of $T$ and $n.$
We deduce that
$$
\|y_n(T)-y^d\| <\frac{\epsilon}{2},
$$
whenever
\begin{equation}\label{n**T}
0<T< \frac{\epsilon}{2M_4}\cdot
\end{equation}
Finally, we can observe that    $F_n$ depends implicitly via $y(t)$ on $T, $ but $n$ is independent of $T$. Then taking $n$ and $T$, respectively, such that $$
2(\frac{ M_2 e^{\|a\|_{L^\infty (\Omega)}}}{n\epsilon^3})^{\frac{1}{2}} <\frac{\epsilon}{2M_4}
$$
and
$$
2 (\frac{ M_2  e^{\|a\|_{L^\infty (\Omega)}}}{n\epsilon^3})^{\frac{1}{2}} < T <\inf(1,\frac{\epsilon}{2M_4}),
$$
so that (\ref{n*T}) and (\ref{n**T}) hold. Hence, we have
$$
\|y(T)-y^d\|\le \|y_n(T)-y(T)\|+ \|y_n(T)-y^d\| <\epsilon\cdot
$$

\subsubsection{The case $a\in W^{2,\infty}(\Omega)$ and $y_0\in L^2(\Omega)$}

For all $\lambda>0,$ we set  $\tilde{y}_{0\lambda}:=\lambda R(\lambda;A)y_{0}\in {\mathcal D}(A),$ and let  $\tilde{y}_{\lambda}$ be the mild solution to (\ref{Eq}) corresponding to the initial state  $\tilde{y}_{0\lambda}$ with   the  $\lambda-$independent control  $v(x,t)=\frac{1}{T}  \ln(\frac{y^d}{y_0}){\bf 1}_{\Lambda\cap O}$.
We have
\begin{equation}\label{0}
  \|y(T)-y^d\|\le \|y(T)-\tilde{y}_{\lambda}(T)\| + \|\tilde{y}_{\lambda}(T)-e^{a(x)} \tilde{y}_{0\lambda}\|+\|e^{a(x)} \tilde{y}_{0\lambda} - y^d\|\cdot
\end{equation}
It follows from the variation of constants formula that
$$
\tilde{y}_{\lambda}(t)-y(t)=S(t)\tilde{y}_{0\lambda}-S(t)y_0+\frac{1}{T} \int_0^t S(t-s) \bigg ( a(x) (\tilde{y}_{\lambda}(s)-y(s)) + f(s,\tilde{y}_{\lambda}(s))-f(s,y(s)) \bigg )ds\cdot
$$
Then, using  the contraction property of the semigroup $S(t)$, it comes
$$
\|\tilde{y}_{\lambda}(t)-y(t)\|\le \|\tilde{y}_{0\lambda}-y_0\|+\frac{\|a\|_{L^\infty(\Omega)}}{T} \int_0^t \|\tilde{y}_{\lambda}(s)-y(s)\| ds+L\int_0^t \|\tilde{y}_{\lambda}(s)-y(s)\| ds, \; \forall t\in [0,T]\cdot
$$
Thus, the Gronwall's inequality gives
$$
\|\tilde{y}_{\lambda}(T)-y(T)\|\le C \|\tilde{y}_{0\lambda}-y_0\|, \; (C=C(\|a\|_{L^\infty(\Omega)}))\cdot
$$
 Moreover, we have
  $$
 \|e^{a(x)} \tilde{y}_{0\lambda} - y^d\|\le e^{\|a\|_{L^\infty(\Omega)}} \|\tilde{y}_{0\lambda}-y_0\|,
 $$
we deduce that there is a $\lambda>0$, which is independent of $T\in (0,1), $ such that
\begin{equation}\label{1}
\|\tilde{y}_{\lambda}(T)-y(T)\| + \|e^{a(x)} \tilde{y}_{0\lambda} - y^d\|<\frac{\epsilon}{2}\cdot
\end{equation}
For such a $\lambda,$ we have: $$y_\lambda(T)-e^{a(x)} \tilde{y}_{0\lambda}=\int_0^Te^{\frac{T-s}{T}a(x)} \big ( A y_\lambda(s) + f(s,y_\lambda(s)) \big ) ds.$$
According to the case discussed in the previous subsection, there exists $0<T<1$ for which
\begin{equation}\label{2}
\|y_{\lambda}(T)-e^{a(x)} \tilde{y}_{0\lambda}\|<\frac{\epsilon}{2}\cdot
\end{equation}
From (\ref{0})-(\ref{2}), we conclude that
$$
\|y(T)-y^d\| <\epsilon\cdot
$$

\subsubsection{The general case: $a\in L^\infty(\Omega)$ and $y_0\in L^2(\Omega)$}

By Lemma \ref{lemma2},  there exists  $(h_r)\subset {\mathcal C}^{\infty}(\mathbf{R}^d)$ such that   $h_r|_\Omega \to h:=e^a$  in $ L^2(\Omega), $ as $r \to 0^+$ and $ h_r>0 $ for  a.e. in $\overline{\Omega}$.  Moreover, since $ e^a \in L^\infty(\Omega)$,   the sequence $(h_r)$ can be chosen such that  $\;(h_r|_\Omega)$ is   bounded in $\Omega$ uniformly w.r.t $r>0$. \\
 Let us define the function: $a_r=\ln(h_r) \in {\mathcal C}^{\infty}(\overline{\Omega}).$
   Since $a_r\in L^\infty(\Omega), $ there is a unique mild solution $y(t)$  of (\ref{Eq}) corresponding to the control $v(x,t)=\displaystyle\frac{a_r(x)}{T}$  and  initial state $y(0)=y_0,$ and $y$  depends continuously on $y_0.$
Let $(y_{0s})\in L^\infty(\Omega)$ be such that $y_{0s} \to y_0$ in $L^2(\Omega), $ as $s\to 0^+.$ Then we have
$$
\|y(T)-y^d\| \le \|y(T)-e^{a_r} y_0\|  + \|e^{a_r} y_0- e^{a_r} y_{0s}\|  + \|e^{a_r} y_{0s}-e^{a} y_{0s}\|   + \|e^{a} y_{0s}-e^{a} y_{0}\|\cdot
$$
Moreover,
$$
\|e^{a_r} y_0- e^{a_r} y_{0s}\| + \|e^{a} y_{0s}-e^{a} y_{0}\|\le \big ( \sup_{r>0} \|e^{a_r}\|_{L^\infty (\Omega)}  + e^{\|a\|_{L^\infty (\Omega)}} \big )   \|y_{0s}-y_0\|\cdot
$$
Since $(a_r)$ is uniformly bounded w.r.t $r$, there exists $s>0$ be such that
$$
\|e^{a_r} y_0- e^{a_r} y_{0s}\|     + \|e^{a} y_{0s}-e^{a} y_{0}\| <\frac{\epsilon}{3},
$$
and for such value of $s,$ we consider a $r>0$  such that
$$
\|e^{a_r} -e^{a}\| \| y_{0s}\|_{L^\infty(\Omega)} <  \frac{\epsilon}{3}\cdot
$$
Finally, for this value of $r$, it comes from the case of the previous subsection that there exists $T>0$ such that $$
\|y(T)-e^{a_r} y_0\|  < \frac{\epsilon}{3}.
$$
Hence we have $\|y(T)-y^d\| < \epsilon\cdot$

\subsection{Proof of  Corollary \ref{cor1}}

We know from Lemma \ref{lemma2} that there exists $(h_r)\subset {\mathcal C}^{\infty}(\mathbf{R}^d)$  such that for all $r>0, $ we have $ h_r>0, $ a.e. $\overline{\Omega}$ and $h_r|_\Omega \to h$ in $L^2(\Omega), $ as $r \to 0^+.$
Let $\epsilon>0$ be fixed, and let $r>0$ be such that
  $$\|h_r -h\|\|y_0\|_{L^\infty(\Omega)}<\frac{\epsilon}{2}\cdot$$
Using the control $v(x,t)=\displaystyle \frac{a_r(x)}{T_1} $ with $0<T_1=T_1(\epsilon,y^d,y_0)<T$ is small enough and $ a_r := \ln(h_r)\in {\mathcal C}^{\infty}(\overline{\Omega}),$
  we get from the proof of Theorem \ref{thma}
$$
\|y(T_1)-e^{a_r} y_0\| < \frac{\epsilon}{2}\cdot
$$
Hence
$$
\begin{array}{lll}
  \|y(T_1)-y^d\| & \le & \|y(T_1)-h_r y_0\| + \|h_r y_0-h y_0\|
 \\
   & \le &  \|y(T_1)-e^{a_r} y_0\| + \|h_r -h \| \, \|y_0\|_{L^\infty(\Omega)}
 \\
  & < &  \epsilon.
\end{array}
$$
This completes the proof.

\subsection{Proof of Corollary \ref{cor2}}

The idea here consists on looking for a control that make the  system (\ref{Eq}) in the form of a bilinear system. Indeed, let us observe that the system (\ref{Eq}) (with $O=\Omega$) can be  written  (at least formally) as follows
\begin{equation}\label{PRE}
\left\{
  \begin{array}{ll}
    y_{t}=\Delta y + (v(x,t)+\frac{f(t,y)}{y})y, & \mbox{in}\;  Q_T  \\
   y(0,t) = 0, & \mbox{on}\;  \Sigma_T      \\
    y(x,0) = y_{0}(x), & \mbox{in}\;  \Omega
\end{array}
\right.
\end{equation}
This leads us to consider the following bilinear system
\begin{equation}\label{pre3}
\left\{
  \begin{array}{ll}
\varphi_{t}=\Delta \varphi + q(x,t)\varphi, & \mbox{in}\;  Q_T \\
\varphi(0,t) = 0, & \mbox{on}\;  \Sigma_T      \\
\varphi(x,0) = y_{0}(x), & \mbox{in}\;  \Omega
\end{array}
\right.
\end{equation}
which is the homogeneous version of the system (\ref{Eq}) (i.e. $f=0$).\\
Let $T, \epsilon>0$ be fixed.  According to Theorem \ref{thma}, applied for $f= 0$, there are   $0<T_1=T_1(y_0,y^d,\epsilon)<T$ and a static control $q(x,t)=q_{1}(x)\in L^\infty(\Omega)$   such that the corresponding state $\varphi$ to system (\ref{pre3}) satisfies the following estimate
 \begin{equation}
 \label{*}
   \|\varphi(T_1)-y^d\| <  \epsilon\cdot
 \end{equation}
(a)  Assume  that: $y^d\in H^2(\Omega) $ and $   g(x) :=-\frac{\Delta y^d}{y^d} {\bf 1}_\Lambda \in L^\infty(\Omega).$\\
Let us consider the following control
\begin{equation}\label{q(x,t)}
  q(x,t)=\left\{
  \begin{array}{ll}
    q_1(x), & t\in [0,T_1) \\
    \\
   g(x), & t\in [T_1,T]
  \end{array}
\right.
\end{equation}
In other words; $q(x,\cdot)=q_1(x) {\bf 1}_{[0,T_1)}+   g(x) {\bf 1}_{[T_1,T]}.$\\
Let us observe that under the assumption (ii) of Theorem \ref{thma}, we have $\Lambda=\{x \in \Omega/ y_0(x)\ne 0 \}= \{x \in \Omega/ y^d(x)\ne 0 \}$. Moreover,  having in mind that $g\in L^\infty(\Omega)$ and that $\Delta y^d =0$  a.e. in $\{x \in \Omega/ y^d(x)= 0 \}$ (see \cite{att}, pp. 210-211), we deduce that $\Delta y^d (x) + g(x) y^d(x) =0,\; a.e.\;  x\in\;  \Omega$ (see also \cite{ouz16}). Hence, with the control $q(x,t)=g(x), \; t\in (T_1,T]$  the state $y^d$ becomes  an equilibrium for the following system
$$\left\{
  \begin{array}{ll}
\phi_{t}=\Delta \phi + g(x)\phi, & \mbox{in}\;  \Omega \times (T_1,T)  \\
\\
\phi(T_1)=y^d, & \mbox{in}\;  \Omega
\end{array}
\right.
$$
In other words, we have
$$
y^d=S(t-T_1)y^d+\int_{T_1}^t S(t-s) (g(x) y^d) ds, \; t\in [T_1,T]\cdot
$$
Thus for all $t\in [T_1,T]$, we have
$$
\varphi(t)-y^d=S(t-T_1) (\varphi(T_1)-y^d ) +\int_{T_1}^t  S(t-s) \bigg ( g(x) (\varphi(s)-y^d ) \bigg ) ds,\; t\in [T_1,T]\cdot
 $$
Then using the fact that $S(t)$ is a contraction semigroup, $ g \in L^\infty(\Omega)$ and $f$ is $L-$Lipschitz, it comes
$$
\|\varphi(t)-y^d\| \le \|\varphi(T_1)-y^d \| + \|g\|_{L^\infty (\Omega)} \int_{T_1}^t \|\varphi(s)-y^d \| ds,\; \forall t\in [T_1,T]\cdot
 $$
Thus Gronwall inequality yields
$$\|\varphi(T)-y^d\| \le e^{T\|g\|_{L^\infty(\Omega)}}   \| \varphi(T_1)-y^d \|. $$
This together with (\ref{*}) gives the approximate steering for the system (\ref{pre3}) at time $T.$\\
Now let $\varphi$ be the solution of (\ref{pre3}) corresponding to the steering control $q$ defined by (\ref{q(x,t)}), and let us return to the whole system (\ref{Eq}), which we intend to excite by the following control
\begin{equation}\label{vGlob}
v(x,t)=q(x,t)-\displaystyle\frac{f(t,\varphi)}{\varphi} {\bf 1}_E,
\end{equation}
where $E=\{(x,t)\in Q_T:\; \varphi(x,t)\ne0\}.$ This leads us  to study the  following  system
\begin{equation}\label{pre2}
\left\{
  \begin{array}{ll}
y_{t}=\Delta y + (q(x,t)-\frac{f(t,\varphi)}{\varphi}{\bf 1}_E)y + f(t,y), & \mbox{in}\;  Q_T  \\
y(0,t) = 0, & \mbox{on}\;  \Sigma_T       \\
y(x,0) = y_{0}(x), & \mbox{in}\;  \Omega.
\end{array}
\right.
\end{equation}
By assumption, we have that for every $y\in L^2(\Omega); \;  |f(t,y)(x)|\le C |y(x)|,\; $ for all $t\in (0,T)$ and for a.e. $\; x\in \Omega$.
Then, it is apparent that $\varphi$ is  a solution of (\ref{pre2}) and is such that $\displaystyle\frac{f(t,\varphi)}{\varphi}{\bf 1}_E \in L^{\infty}(Q_T)$. Hence, by uniqueness (observe that $v\in L^\infty(Q_T)$) we have that $y=\varphi$ is the unique solution of (\ref{pre2}), and hence   $y(T)=\varphi(T)$ approaches $ y^d$ with any a priori fixed  precision.\\

(b)  Assume that:  $y^d \in L^2(\Omega) $ and $y^d\ge 0,\; $ a.e. $x\in\Omega$.\\

Since $y^d\ge 0,$ there exists  $y_\epsilon \in C^\infty (R^d) $ such that $y_\epsilon >0,\; $ a.e. $x\in  \bar{\Omega}$  and
$$\|y^d-y_\epsilon\|< \frac{\epsilon}{2}.$$
We have: $y_\epsilon >\alpha > 0,\; $ a.e. $x\in\Omega$ (where $\alpha=\displaystyle \sup_{x\in \overline{\Omega}} y_\epsilon (x)$). Thus $|\frac{\Delta y_\epsilon }{y_\epsilon}| \le \frac{|\Delta y_\epsilon|}{\alpha},\; a.e.\;  in \;\;\Omega.$\\
This together with the fact that $\Delta y_\epsilon $ is continuous in the bounded set $\bar{\Omega}$, implies that
$\frac{\Delta y_\epsilon }{y_\epsilon}\in L^\infty(\Omega)$. Hence, from the Case (a), there is a control $v_\epsilon\in L^\infty(Q_T)$ such that:
$$\|y(T)-y_\epsilon\|<\frac{\epsilon}{2}.$$
Then we have
$$\|y(T)-y^d\|\le \|y(T)-y_\epsilon\|+ \|y_\epsilon-y^d\|<\epsilon,$$
which achieves the proof.

\end{document}